\documentclass{article}
\usepackage{amssymb}
\usepackage{amsmath}

\newtheorem{theorem}{Theorem}
\newtheorem{corollary}{Corollary}
\newtheorem{lemma}{Lemma}
\newtheorem{remark}{Remark}
\newenvironment{proof}[1][Proof]{\noindent\textbf{#1.} }{\ \rule{0.5em}{0.5em}}

\begin{document}

\title{\textbf{Pointwise strong and very strong\ approximation by Fourier
series of integrable functions}}
\author{\textbf{W\l odzimierz \L enski } \\
University of Zielona G\'{o}ra\\
Faculty of Mathematics, Computer Science and Econometrics\\
65-516 Zielona G\'{o}ra, ul. Szafrana 4a\\
P O L A N D\\
W.Lenski@wmie.uz.zgora.pl }
\date{}
\maketitle

\begin{abstract}
We will present an estimation of the $H_{k_{0},k_{r}}^{q}f$ $\ $and $%
H_{u}^{\lambda \varphi }f$ means as a approximation versions \ of \ the
Totik type generalization$\left( \text{see \cite{11, 12}}\right) $ \ of the
results of \ J. Marcinkiewicz and A. Zygmund in \cite{JM, ZA}. As a measure
of such approximations we will use the function constructed on the base of
definition of the Gabisonia points \cite{1}. Some results on the norm
approximation will also given.

\ \ \ \ \ \ \ \ \ \ \ \ \ \ \ \ \ \ \ \ 


\ \ \ \ \ \ \ \ \ \ \ \ \ \ \ \ \ \ \ 

\end{abstract}

\footnotetext[1]{%
}

\footnotetext[2]{%
}

\newpage

\section{Introduction}

Let $L^{p}\ (1<p<\infty )\;[resp.C]$ be the class of all $2\pi $--periodic
real--valued functions integrable in the Lebesgue sense with $p$--th power [%
\text{continuous]} over $Q=$ $[-\pi ,\pi ]$ and let $X=X^{p}$ where $%
X^{p}=L^{p}$ when $1<p<\infty $ or $X^{p}=C$ when $p=\infty $. Let us define
the norm of $f\in X^{p}$ as\ \ \ \ 
\begin{equation*}
\Vert f\Vert _{_{X^{p}}}=\Vert f(x)\Vert _{_{X^{p}}}=\left\{ 
\begin{array}{ll}
\left( \int_{_{_{Q}}}\mid f(x)\mid ^{p}dx\right) ^{1/p} & when\text{ }%
1<p<\infty \ , \\ 
\sup_{x\in Q}\mid f(x)\mid & \text{ }when\text{ \ }p=\infty .%
\end{array}%
\right. \ \ \ \ 
\end{equation*}

Consider the trigonometric Fourier series\ 
\begin{equation*}
Sf(x)=\frac{a_{o}(f)}{2}+\sum_{k=1}^{\infty }(a_{k}(f)\cos kx+b_{k}(f)\sin
kx)
\end{equation*}%
and denote by $S_{k}f$ \ the partial sums of $Sf$. Then, 
\begin{equation*}
H_{k_{0},k_{r}}^{q}\left( x\right) :=\left\{ \frac{1}{r+1}\sum_{\nu
=0}^{r}\left\vert S_{k_{\nu }}f\left( x\right) -f\left( x\right) \right\vert
^{q}\right\} ^{\frac{1}{q}},\text{ \ \ \ \ \ \ \ \ \ }\left( q>0\right) .
\end{equation*}%
where \ \ $0\leq k_{0}<k_{1}<k_{2}<...<k_{r}$ \ $\left( \geq r\right) ,$ \
and%
\begin{equation*}
H_{u}^{\lambda \varphi }f\left( x\right) :=\left\{ \sum_{\nu =0}^{\infty
}\lambda _{\nu }\left( u\right) \varphi \left( \left\vert S_{\nu }f\left(
x\right) -f\left( x\right) \right\vert \right) \right\} ,\text{ \ \ \ \ }
\end{equation*}%
where \ $\left( \lambda _{\nu }\right) $ \ is a sequence of positive
functions defined on the set having at least one limit point and a function
\ $\varphi :[0,\infty )\rightarrow \mathbf{R.}$

As a measure of the above deviations we will use the pointwise
characteristic, constructed on the base of definition of the Gabisonia
points $(G_{p,s}-points)$, introduced in \cite{1} as follows%
\begin{equation*}
G_{x}f\left( \delta \right) _{p,s}:=\left\{ \sum_{k=1}^{\left[ \pi /\delta %
\right] }\left( \frac{1}{k\delta }\int_{\left( k-1\right) \delta }^{k\delta
}\left\vert \varphi _{x}\left( t\right) \right\vert ^{p}dt\right)
^{s/p}\right\} ^{1/s}
\end{equation*}%
\begin{equation*}
G_{x}^{\circ }f\left( \gamma \right) _{p,s}:=\sup_{0<\delta \leq \gamma
}\left\{ \sum_{k=1}^{\left[ \pi /\delta \right] }\left( \frac{1}{k\delta }%
\int_{\left( k-1\right) \delta }^{k\delta }\left\vert \varphi _{x}\left(
t\right) \right\vert ^{p}dt\right) ^{s/p}\right\} ^{1/s}
\end{equation*}%
and, constructed on the base of definition of the Lebesgue points $%
(L^{p}-points),$ defined as usually%
\begin{eqnarray*}
w_{x}f(\delta )_{p} &:&=\left\{ \frac{1}{\delta }\int_{0}^{\delta
}\left\vert \varphi _{x}\left( t\right) \right\vert ^{p}dt\right\} ^{1/p}, \\
\text{where \ \ }\varphi _{x}\left( t\right) &:&=f\left( x+t\right) +f\left(
x-t\right) -2f\left( x\right) .
\end{eqnarray*}%
\ \ \ We can observe that, for any \ $s>0$ and $p\in \lbrack 1,\infty )$ 
\begin{equation*}
w_{x}f(\delta )_{p}\leq G_{x}f\left( \delta \right) _{p,s}\text{ ,}
\end{equation*}%
for $p\in \lbrack 1,\infty )$\ and $f\in C$%
\begin{equation*}
w_{x}f(\delta )_{p}\leq \omega _{C}f\left( \delta \right) .
\end{equation*}%
By the Minkowski inequality, with \ $\widetilde{p}\geq s>p\geq 1$ for $f\in
X^{\widetilde{p}}$, 
\begin{equation*}
\Vert G_{\cdot }f\left( \delta \right) _{p,s}\Vert _{_{X^{\widetilde{p}%
}}}\leq \omega _{_{X^{\widetilde{p}}}}f\left( \frac{\left\vert \log \left[
\pi /\delta \right] \right\vert }{\left( \pi /\delta \right) ^{\frac{1}{p}-%
\frac{1}{s}}}\right) \text{ \ \ }\left( \text{cf. \cite{1}}\right)
\end{equation*}%
and%
\begin{equation*}
\Vert w_{\cdot .}f(\delta )_{p}\Vert _{_{X^{\widetilde{p}}}}\leq \omega
_{_{X^{\widetilde{p}}}}f\left( \delta \right) ,
\end{equation*}%
where $\omega _{X}f$ is the modulus of continuity of $f$ in the space $X=X^{%
\widetilde{p}}$ defined by the formula 
\begin{equation*}
\omega _{X}f\left( \delta \right) :=\sup_{0<\left\vert h\right\vert \leq
\delta }\left\Vert \varphi _{\cdot }\left( h\right) \right\Vert _{X}\text{ \
.}
\end{equation*}

It is well-known that $H_{0,r}^{q}f\left( x\right) -$ means tend to $0$ $%
\left( \text{as}\ r\rightarrow \infty \right) $ \ at the $L^{p}-points$ $x$
\ of $f$ $\in L^{p}$ $\left( 1<p\leq \infty \right) $ and at the $%
G_{1,s}-points$ $x$ \ of $f$ $\in L^{1}$ $\left( s>1\right) .$ These facts
were proved as a generalization of the Fej\'{e}r classical result on the
convergence of the $\left( C,1\right) $ -means of Fourier series by G. H.
Hardy, J. E. Littlewood in \cite{2}.and by O. D. Gabisonia in \cite{1}. In
case \ $L^{1}$ and convergence almost everywhere the first results on this
area belong to J. Marcinkiewicz \cite{JM}\ and A. Zygmund \cite{ZA}. The
estimates of $H_{0,r}^{q}f\left( x\right) $ -mean were obtained in \cite{3,
4, 10}. Here we present estimations of the $H_{k_{0},k_{r}}^{q}\left(
x\right) $ and $H_{\upsilon }^{\lambda \varphi }f\left( x\right) $ means as
approximation versions \ of \ the Totik type $\left( \text{see \cite{11, 12}}%
\right) $ \ generalization of the result of \ O. D. Gabisonia \cite{1}. We
also give some corollaries on norm approximation.

By $K$ we shall designate either an absolute constant or a constant
depending on the some parameters, not necessarily the same of each
occurrence. We shall write $I_{1}\ll I_{2}$ if there exists a positive
constant $K$ , sometimes depended on some parameters, such that $I_{1}\leq
KI_{2}$.\medskip

\section{Statement of the results}

Let us consider a function $w_{x}$ of modulus of continuity type on the
interval $[0,+\infty ),$ i.e. a nondecreasing continuous function having the
following properties: \ $w_{x}\left( 0\right) =0,$ \ $w_{x}\left( \delta
_{1}+\delta _{2}\right) \leq w_{x}\left( \delta _{1}\right) +w_{x}\left(
\delta _{2}\right) $ \ for any \ \ $0\leq \delta _{1}\leq \delta _{2}\leq
\delta _{1}+\delta _{2}$ \ and let 
\begin{equation*}
L^{p}\left( w_{x}\right) _{s}=\left\{ f\in L^{p}:\text{ }G_{x}f\left( \delta
\right) _{p,s}\leq w_{x}\left( \delta \right) \text{ \ , \ where \ }\delta
>0,\text{ \ }s>p\geq 1\right\} .
\end{equation*}

In the same way let 
\begin{equation*}
X^{p}\left( \omega \right) _{s}=\left\{ f\in X^{p}:\left\Vert G_{\cdot
}^{\circ }f\left( \delta \right) _{1,s}\right\Vert _{X^{p}}\leq \omega
\left( \delta \right) ,\text{ \ with a modulus of continuity \ }\omega
\right\}
\end{equation*}%
We start with theorems:

\begin{theorem}
If \textit{\ }$f\in L^{1}\left( w_{x}\right) _{s}$ \textit{\ and }\ $0\leq
k_{0}<k_{1}<k_{2}<...<k_{r}$ \ $\left( \geq r\right) $ ,\ \textit{then} 
\begin{equation*}
H_{k_{0},k_{r}}^{q^{\prime }}\left( x\right) \ll w_{x}\left( \frac{\pi }{%
k_{0}+1}\right) \log \frac{k_{r}+1}{r+1/2}\text{ \ ,}
\end{equation*}%
where \ $0<q^{\prime }\leq q$ $\ \left( \geq 2\right) \ $such $\ $that$\ \ 
\frac{1\text{ }}{s}+\frac{1}{q}=1$ $.$
\end{theorem}

\begin{theorem}
\textit{If }$f\in X^{p}$\textit{\ \ } \textit{and }\ $0\leq
k_{0}<k_{1}<k_{2}<...<k_{r}$ \ $\left( \geq r\right) $ ,\ \textit{then} 
\begin{equation*}
\left\Vert H_{k_{0},k_{r}}^{q^{\prime }}f\left( \cdot \right) \right\Vert
_{X^{p}}\ll \omega \left( \frac{\pi }{k_{0}+1}\right) \log \frac{k_{r}+1}{%
r+1/2}\text{ \ ,}
\end{equation*}%
where \ $0<q^{\prime }\leq q$ $\ \left( \geq 2\right) \ $such $\ $that$\ \ 
\frac{1\text{ }}{s}+\frac{1}{q}=1$ $.$
\end{theorem}

Denoting%
\begin{equation*}
\Phi =\left\{ 
\begin{array}{c}
\varphi :\varphi \left( 0\right) =0,\text{ \ }\varphi \nearrow ,\text{ \ }%
\varphi \left( 2u\right) \ll \varphi \left( u\right) \text{ \ for }u\in
\left( 0,1\right) \\ 
\text{and \ }\log \varphi \left( u\right) =O\left( u\right) \text{ \ as \ }%
u\rightarrow \infty%
\end{array}%
\right\}
\end{equation*}%
we can formulate the next theorems on the base of the before two.

\begin{theorem}
\textit{If }$f\in L^{1}$ \textit{\ , }$\varphi \in \Phi $ \ \textit{and }\ $%
\lambda _{\nu }\left( m\right) =\frac{1}{N_{m}+1}$ $\ $for $\ \nu $ $%
=N_{m-2}+1,N_{m-2}+2,...,N_{m}$ \ and \ $\lambda _{\nu }\left( m\right) =0$
\ otherwise, \ \textit{then}%
\begin{equation*}
H_{m}^{\lambda \varphi }f\left( x\right) \ll \varphi \left( w_{x}\left( 
\frac{\pi }{N_{m-2}+1}\right) \right) \text{ \ ,}
\end{equation*}%
where $m=1.2,..$ and \ $s>1.$
\end{theorem}

\begin{theorem}
\textit{If }$f\in X^{p}$\textit{\ , }$\varphi \in \Phi $ \ \textit{and }\ $%
\lambda _{\nu }\left( m\right) =\frac{1}{N_{m}+1}$ $\ $for $\ \nu $ $%
=N_{m-2}+1,N_{m-2}+2,...,N_{m}.$and \ $\lambda _{\nu }\left( m\right) =0$ \
otherwise, \ \textit{then}%
\begin{equation*}
\left\Vert H_{m}^{\lambda \varphi }f\left( \cdot \right) \right\Vert
_{X^{p}}\ll \varphi \left( \omega \left( \frac{\pi }{N_{m-2}+1}\right)
\right) \text{ \ ,}
\end{equation*}%
where $m=1.2,..$ and \ $s>1.$
\end{theorem}

Let, as in the Leindler monograph \cite{L} p.15,%
\begin{eqnarray*}
\Lambda _{\tau }\left( N_{m}\right) &=&\left\{ \left( \lambda _{\nu }\right)
:\left( \frac{1}{N_{m}}\sum_{\nu =N_{m-2}+1}^{N_{m}}\left( \lambda _{\nu
}\right) ^{\tau }\right) ^{1/\tau }\ll \left( \frac{1}{N_{m}}\sum_{\nu
=N_{m-2}+1}^{N_{m}}\lambda _{\nu }\right) \right. \text{ \ } \\
&&\left. \text{for \ }s\geq 1\text{ \ and }N_{m}<N_{m+1}\text{ , }N_{0}=0,%
\text{ \ }N_{-1}=-1\right\} .
\end{eqnarray*}

Finally, we present very general results deduced from the above theorems.

\begin{theorem}
If \textit{\ }$f\in L^{1}$ \textit{\ \ then } for \ $\left( \lambda _{\nu
}\right) \in \Lambda _{\tau }\left( N_{m}\right) $ with $\tau >1$ \ and \
for $\varphi \in \Phi ,$ we have%
\begin{equation*}
H_{u}^{\lambda \varphi }f\left( x\right) \ll \sum_{m=1}^{\infty }\sum_{\nu
=N_{m-2}+1}^{N_{m}}\lambda _{\nu }\left( u\right) \varphi \left( w_{x}\left( 
\frac{\pi }{N_{m-2}+1}\right) \right) \text{ \ ,}
\end{equation*}%
for any real \ $u$ \ and \ $s>1.$
\end{theorem}

\begin{theorem}
\textit{If }$f\in X^{p}$\textit{\ \ then}, for \ $\left( \lambda _{\nu
}\right) \in \Lambda _{\tau }\left( N_{m}\right) $ with $\tau >1$\ and \ for 
$\varphi \in \Phi ,$ we have%
\begin{equation*}
\left\Vert H_{u}^{\lambda \varphi }f\left( \cdot \right) \right\Vert
_{X^{p}}\ll \sum_{m=1}^{\infty }\sum_{\nu =N_{m-2}+1}^{N_{m}}\lambda _{\nu
}\left( u\right) \varphi \left( \omega \left( \frac{\pi }{N_{m-2}+1}\right)
\right) \text{ \ ,}
\end{equation*}%
for any real \ $u$ \ and \ $s>1.$
\end{theorem}

From these theorems we can derive the following corollary.

\begin{corollary}
If we additionally suppose that \ $\lim\limits_{u\rightarrow u_{0}}\lambda
_{\nu }\left( u\right) =0$ \ for all $\nu $ \ and that\ $\sum_{\nu }^{\infty
}\lambda _{\nu }\left( u\right) $ \ converges, then we have%
\begin{equation*}
\lim\limits_{u\rightarrow u_{0}}H_{u}^{\lambda \varphi }f\left( x\right) =0
\end{equation*}%
at every \ $G_{1,s}-points$ \ $x$ \ of the function \ $f,$ and%
\begin{equation*}
\lim\limits_{u\rightarrow u_{0}}\left\Vert H_{u}^{\lambda \varphi }f\left(
\cdot \right) \right\Vert _{L^{p}}=0.
\end{equation*}%
for any real \ $s>1.$
\end{corollary}

\begin{remark}
We can observe that in the light of the Gabisonia result \cite{G} our
pointwise results remain true for $f\in L^{p}$ $\ \left( p>1\right) ,$ since
every $L^{p}-point$ of the function \ $f$ is its $G_{p,s}-point.$
\end{remark}

\section{Auxiliary results}

At the begin we present some lemmas on pointwise characteristics.

\begin{lemma}
$\left( \text{Property 1 \cite{10}}\right) $ If $f\in L^{p}$ $\ \left( p\geq
1\right) $ and \ $\lambda ,\beta >0,$ then%
\begin{equation*}
\left\{ \lambda ^{\beta }\int_{\lambda }^{\pi }t^{-\left( \beta +1\right)
}\left\vert \varphi _{x}\left( t\right) \right\vert ^{p}dt\right\} ^{1/\beta
}\ll G_{x}f\left( \lambda \right) _{p,s}
\end{equation*}%
with $s>p$ \ such that \ $s\left( 1-\beta \right) <p.$
\end{lemma}

\begin{lemma}
$\left( \text{Property 2 \cite{10}}\right) $ If $f\in L^{p}$ $\ \left( p\geq
1\right) $ and \ $\lambda ,\beta >0,$ then%
\begin{equation*}
G_{x}f\left( 2\lambda \right) _{p,s}\leq 2^{1/p-1/s}G_{x}f\left( \lambda
\right) _{p,s}
\end{equation*}%
with $s>p.$

\begin{lemma}
If $f\in L^{p}$ $\ \left( p\geq 1\right) ,$ then%
\begin{eqnarray*}
&&\left\{ \frac{1}{\delta }\int_{0}^{\delta }\left\vert \varphi _{x}\left(
t+\gamma \right) -\varphi _{x}\left( t\right) \right\vert ^{p}dt\right\}
^{1/p} \\
&\leq &\left( 2^{1/p}+4^{1/p}\right) w_{x}f\left( 2\delta \right) \leq
\left( 2^{1/p}+4^{1/p}\right) G_{x}f\left( \delta \right) _{p,s}
\end{eqnarray*}%
for any positive \ $\gamma \leq \delta $ \ and \ $1\leq p<s.$
\end{lemma}
\end{lemma}

\begin{proof}
Since \ $\gamma \leq \delta $ we have%
\begin{eqnarray*}
&&\left\{ \frac{1}{\delta }\int_{0}^{\delta }\left\vert \varphi _{x}\left(
t\pm \gamma \right) -\varphi _{x}\left( t\right) \right\vert ^{p}dt\right\}
^{1/p} \\
&\leq &\left\{ \frac{1}{\delta }\int_{0}^{\delta }\left\vert \varphi
_{x}\left( t\right) \right\vert ^{p}dt\right\} ^{1/p}+\left\{ \frac{1}{%
\delta }\int_{\pm \gamma }^{\delta \pm \gamma }\left\vert \varphi _{x}\left(
t\right) \right\vert ^{p}dt\right\} ^{1/p}
\end{eqnarray*}%
\begin{eqnarray*}
&\leq &\left\{ \frac{1}{\delta }\int_{0}^{\delta }\left\vert \varphi
_{x}\left( t\right) \right\vert ^{p}dt\right\} ^{1/p}+\left\{ \frac{1}{%
\delta }\int_{-2\delta }^{2\delta }\left\vert \varphi _{x}\left( t\right)
\right\vert ^{p}dt\right\} ^{1/p} \\
&\leq &\left\{ \frac{2}{2\delta }\int_{0}^{2\delta }\left\vert \varphi
_{x}\left( t\right) \right\vert ^{p}dt\right\} ^{1/p}+\left\{ \frac{2}{%
\delta }\int_{0}^{2\delta }\left\vert \varphi _{x}\left( t\right)
\right\vert ^{p}dt\right\} ^{1/p}
\end{eqnarray*}%
\begin{equation*}
=\left( 2^{1/p}+4^{1/p}\right) \left\{ \frac{1}{2\delta }\int_{0}^{2\delta
}\left\vert \varphi _{x}\left( t\right) \right\vert ^{p}dt\right\} ^{1/p}
\end{equation*}%
and our inequalities are evident.
\end{proof}

Under the notation 
\begin{equation*}
\Phi _{x}f\left( \delta ,\gamma \right) :=\frac{1}{\delta }\int_{\gamma
}^{\gamma +\delta }\varphi _{x}\left( t\right) dt\text{, \ \ \ \ }%
W_{x}f\left( \delta ,\gamma \right) _{p}:=\left[ \frac{1}{\delta }%
\int_{\gamma }^{\gamma +\delta }\left\vert \varphi _{x}\left( t\right)
\right\vert ^{p}dt\right] ^{1/p}
\end{equation*}%
we can formulate a lemma.

\begin{lemma}
If $f\in L^{p}$ $\ \left( p\geq 1\right) ,$ then%
\begin{equation*}
\left\vert \Phi _{x}f\left( \delta ,\gamma \right) \right\vert \leq
W_{x}f\left( \delta ,\gamma \right) _{p}\ll w_{x}f\left( 2\delta \right)
\end{equation*}%
for any positive \ $\gamma \leq \delta .$
\end{lemma}

\begin{proof}
The first inequality is evidence, then we prove the second one only.

If \ $f\in L^{p},$ then%
\begin{eqnarray*}
&&\left\{ \frac{1}{\delta }\int_{0}^{\delta }\left\vert \varphi _{x}\left(
t+\gamma \right) \right\vert ^{p}dt\right\} ^{1/p}-\left\{ \frac{1}{\delta }%
\int_{0}^{\delta }\left\vert \varphi _{x}\left( t\right) \right\vert
^{p}dt\right\} ^{1/p} \\
&\leq &\left\{ \frac{1}{\delta }\int_{0}^{\delta }\left\vert \varphi
_{x}\left( t+\gamma \right) -\varphi _{x}\left( t\right) \right\vert
^{p}dt\right\} ^{1/p}
\end{eqnarray*}%
whence%
\begin{eqnarray*}
&&\left\{ \frac{1}{\delta }\int_{\gamma }^{\gamma +\delta }\left\vert
\varphi _{x}\left( t\right) \right\vert ^{p}dt\right\} ^{1/p} \\
&\leq &\left\{ \frac{1}{\delta }\int_{0}^{\delta }\left\vert \varphi
_{x}\left( t+\gamma \right) -\varphi _{x}\left( t\right) \right\vert
^{p}dt\right\} ^{1/p}+\left\{ \frac{1}{\delta }\int_{0}^{\delta }\left\vert
\varphi _{x}\left( t\right) \right\vert ^{p}dt\right\} ^{1/p}
\end{eqnarray*}%
and by the previous lemma our second relation follows.
\end{proof}

We will also need the inequalities for norms.

\begin{lemma}
If $f\in L^{p}$ $\ \left( p\geq 1\right) ,$ then%
\begin{equation*}
\left\Vert \Phi _{\cdot }f\left( \delta ,\gamma \right) \right\Vert
_{L^{p}}\leq \left\Vert W_{\cdot }f\left( \delta ,\gamma \right)
_{p}\right\Vert _{L^{p}}\leq 2\omega _{L^{p}}f\left( \delta +\gamma \right)
\end{equation*}%
and%
\begin{equation*}
\left\Vert \left[ \frac{1}{\delta }\int_{0}^{\delta }\left\vert \varphi
_{\cdot }\left( t\right) -\varphi _{\cdot }\left( t\pm \gamma \right)
\right\vert ^{p}dt\right] ^{1/p}\right\Vert _{L^{p}}\leq 2\omega
_{L^{p}}f\left( \gamma \right) \text{ \ ,}
\end{equation*}%
for any positive\ $\gamma $ and $\delta .$
\end{lemma}

\begin{proof}
If \ $f\in L^{p},$ then, by monotonicity of the norm as a functional and by
the above Lemma,%
\begin{equation*}
\left\Vert \Phi _{\cdot }f\left( \delta ,\gamma \right) \right\Vert
_{L^{p}}\leq \left\Vert W_{\cdot }f\left( \delta ,\gamma \right)
_{p}\right\Vert _{L^{p}}
\end{equation*}%
and consequently%
\begin{eqnarray*}
\left\Vert w_{\cdot }f\left( \delta ,\gamma \right) _{p}\right\Vert _{L^{p}}
&=&\left\{ \int_{-\pi }^{\pi }\left[ \frac{1}{\delta }\int_{\gamma }^{\gamma
+\delta }\left\vert \varphi _{x}\left( t\right) \right\vert ^{p}dt\right]
dx\right\} ^{1/p} \\
&=&\left\{ \frac{1}{\delta }\int_{\gamma }^{\gamma +\delta }\left[
\int_{-\pi }^{\pi }\left\vert \varphi _{x}\left( t\right) \right\vert ^{p}dx%
\right] dt\right\} ^{1/p} \\
&\leq &\left\{ \frac{1}{\delta }\int_{\gamma }^{\gamma +\delta }\left[
2\omega _{L^{p}}f\left( t\right) \right] ^{p}dt\right\} ^{1/p} \\
&\leq &2\omega _{L^{p}}f\left( \delta +\gamma \right) ,
\end{eqnarray*}%
whence our first result follows.

In the next one we will change order of integration, whence%
\begin{eqnarray*}
&&\left\Vert \left[ \frac{1}{\delta }\int_{0}^{\delta }\left\vert \varphi
_{\cdot }\left( t\right) -\varphi _{\cdot }\left( t\pm \gamma \right)
\right\vert ^{p}dt\right] ^{1/p}\right\Vert _{L^{p}} \\
&\leq &\left\{ \frac{1}{\delta }\int_{0}^{\delta }\left[ \int_{-\pi }^{\pi
}\left\vert \varphi _{x}\left( t\right) -\varphi _{x}\left( t\pm \gamma
\right) \right\vert ^{p}dx\right] dt\right\} ^{1/p} \\
&\leq &\left\{ \frac{1}{\delta }\int_{0}^{\delta }\left[ \int_{-\pi }^{\pi
}\left( \left\vert f\left( x+t\right) -f\left( x+t\pm \gamma \right)
\right\vert \right. \right. \right. \\
&&+\left. \left. \left. \left\vert f\left( x-t\right) -f\left( x-t\mp \gamma
\right) \right\vert \right) ^{p}dx\right] dt\right\} ^{1/p} \\
&\leq &\left\{ \frac{1}{\delta }\int_{0}^{\delta }\left[ 2\omega
_{L^{p}}f\left( \gamma \right) \right] ^{p}dt\right\} ^{1/p}=2\omega
_{L^{p}}f\left( \gamma \right)
\end{eqnarray*}%
and thus our proof is complete.
\end{proof}

In the sequel we will also need some another lemmas with the next notions.
Let%
\begin{equation*}
\Psi _{x}f\left( \delta ,\gamma \right) _{p}:=\left\{ \frac{1}{\gamma }%
\int_{\gamma }^{\gamma +\delta }\left\vert \varphi _{x}\left( t\right)
\right\vert ^{p}dt\right\} ^{1/p}\text{,}
\end{equation*}%
then we have

\begin{lemma}
If $f\in L^{p}$ $\ \left( p\geq 1\right) ,$ then%
\begin{equation*}
\Psi _{x}f\left( \delta ,\gamma \right) _{p}\ll G_{x}f\left( \delta \right)
_{p,s}
\end{equation*}%
for any positive $\delta $\ $\leq \gamma $ such that $\gamma +\delta \leq
\pi $ $\ $and$\ 1\leq p<s.$
\end{lemma}

\begin{proof}
There exists a natural $k^{\prime }$ such that $\left( k^{\prime }-1\right)
\delta \leq \gamma +\delta \leq k^{\prime }\delta .$ Then

\begin{eqnarray*}
\Psi _{x}f\left( \delta ,\gamma \right) _{p} &\ll &\left( \frac{1}{k^{\prime
}\delta }\int_{\left( k^{\prime }-2\right) \delta }^{k^{\prime }\delta
}\left\vert \varphi _{x}\left( t\right) \right\vert ^{p}dt\right) ^{1/p} \\
&\ll &\left( \frac{1}{k^{\prime }\delta }\int_{\left( k^{\prime }-1\right)
\delta }^{k^{\prime }\delta }\left\vert \varphi _{x}\left( t\right)
\right\vert ^{p}dt\right) ^{1/p}+\left( \frac{1}{\left( k^{\prime }-1\right)
\delta }\int_{\left( k^{\prime }-2\right) \delta }^{\left( k^{\prime
}-1\right) \delta }\left\vert \varphi _{x}\left( t\right) \right\vert
^{p}dt\right) ^{1/p} \\
&\ll &\left\{ \sum_{k=1}^{\left[ \pi /\delta \right] }\left( \frac{1}{%
k\delta }\int_{\left( k-1\right) \delta }^{k\delta }\left\vert \varphi
_{x}\left( t\right) \right\vert ^{p}dt\right) ^{s/p}\right\}
^{1/s}=G_{x}f\left( \delta \right) _{p,s}
\end{eqnarray*}%
and our estimate is proved.

\begin{lemma}
If $f\in L^{p}$ $\ \left( p\geq 1\right) ,$ then%
\begin{equation*}
\left\Vert \Psi _{\cdot }f\left( \delta ,\gamma \right) _{1}\right\Vert
_{L^{p}}\ll \omega _{L^{p}}f\left( \delta \right)
\end{equation*}%
for any positive $\delta $\ $\leq \gamma $ such that $\gamma +\delta \leq
\pi $ $.$
\end{lemma}

\begin{proof}
Easy calculation gives%
\begin{eqnarray*}
\left\Vert \Psi _{\cdot }f\left( \delta ,\gamma \right) _{1}\right\Vert
_{L^{p}} &\ll &\frac{1}{\gamma }\int_{\gamma }^{\gamma +\delta }\omega
_{L^{p}}f\left( t\right) dt\ll \frac{1}{\gamma }\int_{\gamma }^{\gamma
+\delta }\omega _{L^{p}}f\left( \gamma +\delta \right) dt \\
&\ll &\frac{\omega _{L^{p}}f\left( \gamma \right) }{\gamma }\int_{\gamma
}^{\gamma +\delta }dt=\delta \frac{\omega _{L^{p}}f\left( \gamma \right) }{%
\gamma }\ll \delta \frac{\omega _{L^{p}}f\left( \delta \right) }{\delta }
\end{eqnarray*}%
and our Lemma is proved.
\end{proof}
\end{proof}

\section{Proofs of the results}

We only prove Theorems 1, 3 and \ 5 because in the remain proofs we have to
use Lemma 5 and Lemma 7 instead of Lemmas 3, 4 and 6.

\subsection{Proof of Theorem 1}

Let 
\begin{eqnarray*}
H_{k_{0},k_{r}}^{q}\left( x\right) &=&\left\{ \frac{1}{r+1}\sum_{\nu
=0}^{r}\left\vert \frac{1}{\pi }\int_{0}^{\pi }\varphi _{x}\left( t\right)
D_{k_{\nu }}\left( t\right) dt\right\vert ^{q}\right\} ^{1/q} \\
&\leq &A_{r}+B_{r}+C_{r},
\end{eqnarray*}%
where%
\begin{equation*}
A_{r}=\left\{ \frac{1}{r+1}\sum_{\nu =0}^{r}\left\vert \frac{1}{\pi }%
\int_{0}^{2\delta }\varphi _{x}\left( t\right) D_{k_{\nu }}\left( t\right)
dt\right\vert ^{q}\right\} ^{1/q},
\end{equation*}%
\begin{equation*}
B_{r}=\left\{ \frac{1}{r+1}\sum_{\nu =0}^{r}\left\vert \frac{1}{\pi }%
\int_{2\delta }^{2\gamma }\varphi _{x}\left( t\right) D_{k_{\nu }}\left(
t\right) dt\right\vert ^{q}\right\} ^{1/q},
\end{equation*}%
\begin{equation*}
C_{r}=\left\{ \frac{1}{r+1}\sum_{\nu =0}^{r}\left\vert \frac{1}{\pi }%
\int_{2\gamma }^{\pi }\varphi _{x}\left( t\right) D_{k_{\nu }}\left(
t\right) dt\right\vert ^{q}\right\} ^{1/q},
\end{equation*}%
with \ $D_{k_{\nu }}\left( t\right) =\frac{\sin \left( \left( k_{\nu }+\frac{%
1}{2}\right) t\right) }{2\sin \frac{t}{2}}$ , $\delta =\delta _{\nu }$ and $%
\gamma =\gamma _{r}^{2}/\delta _{\nu },\ $putting\ $\ \delta _{\nu }=\frac{%
\pi }{k_{\nu }+1/2}$\ , \ $\gamma _{r}=\frac{\pi }{r+1/2}$. In the case $%
\gamma \geq \pi /2$ we will have $C_{r}\equiv 0$. At the begin%
\begin{eqnarray*}
A_{r} &\leq &\left\{ \frac{1}{r+1}\sum_{\nu =0}^{r}\left[ \frac{k_{\nu }+1}{%
\pi }\int_{0}^{2\delta }\left\vert \varphi _{x}\left( t\right) \right\vert dt%
\right] ^{q}\right\} ^{1/q} \\
&\leq &\left\{ \frac{1}{r+1}\sum_{\nu =0}^{r}\left[ 4\frac{k_{\nu }+1/2}{%
2\pi }\int_{0}^{2\delta _{\nu }}\left\vert \varphi _{x}\left( t\right)
\right\vert dt\right] ^{q}\right\} ^{1/q} \\
&\leq &\left\{ \frac{1}{r+1}\sum_{\nu =0}^{r}\left[ 4w_{x}f(2\delta _{\nu
})_{1}\right] ^{q}\right\} ^{1/q} \\
&\leq &4w_{x}(2\delta _{0})\leq 8w_{x}(\delta _{0}).
\end{eqnarray*}

The terms \ $B_{k_{r}}$ and \ $C_{k_{r}}$ we estimate by the Totik method 
\cite{12}.and its modification from \cite{WL} We divide the term \ $B_{r}$
into the two parts%
\begin{eqnarray*}
B_{r} &=&\left\{ \frac{1}{r+1}\sum_{\nu =0}^{r}\left\vert \frac{1}{\pi }%
\int_{2\delta }^{2\gamma }\varphi _{x}\left( t\right) D_{k_{\nu }}\left(
t\right) dt\right\vert ^{q}\right\} ^{1/q} \\
&\leq &\left\{ \frac{1}{r+1}\left( \sum_{\nu =0}^{\nu _{0}-1}+\sum_{\nu =\nu
_{0}}^{r}\right) \left\vert \frac{1}{\pi }\int_{2\delta }^{2\gamma }\varphi
_{x}\left( t\right) D_{k_{\nu }}\left( t\right) dt\right\vert ^{q}\right\}
^{1/q}
\end{eqnarray*}%
\begin{eqnarray*}
&\leq &\left\{ \frac{1}{r+1}\sum_{\nu =0}^{\nu _{0}-1}\left\vert \frac{1}{%
\pi }\int_{2\gamma }^{2\delta }\varphi _{x}\left( t\right) D_{k_{\nu
}}\left( t\right) dt\right\vert ^{q}\right\} ^{1/q} \\
&&+\left\{ \frac{1}{r+1}\sum_{\nu =\nu _{0}}^{r}\left\vert \frac{1}{\pi }%
\int_{2\delta }^{2\gamma }\varphi _{x}\left( t\right) D_{k_{\nu }}\left(
t\right) dt\right\vert ^{q}\right\} ^{1/q}
\end{eqnarray*}%
\begin{eqnarray*}
&\leq &\left\{ \frac{1}{r+1}\sum_{\nu =0}^{\nu _{0}-1}\left[ \frac{k_{\nu }+1%
}{\pi }\int_{2\gamma }^{2\delta }\left\vert \varphi _{x}\left( t\right)
\right\vert dt\right] ^{q}\right\} ^{1/q}+B_{r,\nu _{0}} \\
&\leq &\left\{ \frac{1}{r+1}\sum_{\nu =0}^{r}\left[ \frac{4}{2\delta _{\nu }}%
\int_{0}^{2\delta _{\nu }}\left\vert \varphi _{x}\left( t\right) \right\vert
dt\right] ^{q}\right\} ^{1/q}+B_{r,\nu _{0}}
\end{eqnarray*}%
\begin{eqnarray*}
&\leq &\left\{ \frac{1}{r+1}\sum_{\nu =0}^{r}\left[ 4w_{x}f(2\delta _{\nu
})_{1}\right] ^{q}\right\} ^{1/q}+B_{r,\nu _{0}} \\
&\leq &8w_{x}(\delta _{0})+B_{r,\nu _{0}}\text{ ,}
\end{eqnarray*}%
where the index \ $\nu _{0}$\ is such that \ $k_{\nu _{0}-1}<r\leq k_{\nu
_{0}}$ $\left( \delta _{\nu _{0}}\leq \gamma _{r}<\delta _{\nu _{0}-1}\text{
with }k_{-1}=0\right) .$\ Next the term \ $B_{r,\nu _{0}}.$ we divide into
the three parts.

\begin{eqnarray*}
&&B_{r,\nu _{0}}. \\
&=&\left\{ \frac{1}{r+1}\sum_{\nu =\nu _{0}}^{r}\left\vert \frac{1}{\pi }%
\int_{2\delta _{\nu }}^{2\gamma }\varphi _{x}\left( t\right) D_{k_{\nu
}}\left( t\right) dt\right\vert ^{q}\right\} ^{1/q} \\
&=&\frac{1}{2}\left\{ \frac{1}{r+1}\sum_{\nu =\nu _{0}}^{r}\left\vert \frac{1%
}{\pi }\left( \int_{2\delta _{\nu }}^{2\gamma }+\int_{\delta _{\nu
}}^{2\gamma -\delta _{\nu }}+\int_{2\gamma -\delta _{\nu }}^{2\gamma
}-\int_{\delta _{\nu }}^{2\delta _{\nu }}\right) \varphi _{x}\left( t\right)
D_{k_{\nu }}\left( t\right) dt\right\vert ^{q}\right\} ^{1/q} \\
&\leq &B_{r,\nu _{0}}^{1}+B_{r,\nu _{0}}^{2}+B_{r,\nu _{0}}^{3},
\end{eqnarray*}%
where the first term%
\begin{eqnarray*}
&&B_{r,\nu _{0}}^{1} \\
&=&\frac{1}{2}\left\{ \frac{1}{r+1}\sum_{\nu =\nu _{0}}^{r}\left\vert \frac{1%
}{\pi }\left( \int_{2\delta _{\nu }}^{2\gamma }+\int_{\delta _{\nu
}}^{2\gamma -\delta _{\nu }}\right) \varphi _{x}\left( t\right) D_{k_{\nu
}}\left( t\right) dt\right\vert ^{q}\right\} ^{1/q} \\
&=&\frac{1}{2}\left\{ \frac{1}{r+1}\sum_{\nu =\nu _{0}}^{r}\left\vert \frac{1%
}{\pi }\int_{2\delta _{\nu }}^{2\gamma }\left[ \varphi _{x}\left( t\right)
D_{k_{\nu }}\left( t\right) +\varphi _{x}\left( t-\delta _{\nu }\right)
D_{k_{\nu }}\left( t-\delta _{\nu }\right) \right] dt\right\vert
^{q}\right\} ^{1/q}
\end{eqnarray*}%
\begin{eqnarray*}
&\leq &\frac{1}{2}\left\{ \frac{1}{r+1}\sum_{\nu =\nu _{0}}^{r}\left\vert 
\frac{1}{\pi }\int_{2\delta _{\nu }}^{2\gamma }\left( \varphi _{x}\left(
t\right) -\varphi _{x}\left( t-\delta _{\nu }\right) \right) D_{k_{\nu
}}\left( t\right) dt\right\vert ^{q}\right\} ^{1/q} \\
&&+\frac{1}{2}\left\{ \frac{1}{r+1}\sum_{\nu =\nu _{0}}^{r}\left\vert \frac{1%
}{\pi }\int_{2\delta _{\nu }}^{2\gamma }\varphi _{x}\left( t-\delta _{\nu
}\right) \left( D_{k_{\nu }}\left( t\right) +D_{k_{\nu }}\left( t-\delta
_{\nu }\right) \right) dt\right\vert ^{q}\right\} ^{1/q}.
\end{eqnarray*}%
Using the partial integration we obtain 
\begin{eqnarray*}
&&B_{r,\nu _{0}}^{1} \\
&\leq &\frac{1}{2}\left\{ \frac{1}{r+1}\sum_{\nu =\nu _{0}}^{r}\left\vert 
\frac{1}{\pi }\int_{2\delta _{\nu }}^{2\gamma }\frac{d}{dt}\left[
\int_{0}^{t}\left( \varphi _{x}\left( u\right) -\varphi _{x}\left( u-\delta
_{\nu }\right) \right) \sin \frac{\left( 2k_{\nu }+1\right) u}{2}du\right] 
\frac{1}{2\sin \frac{t}{2}}dt\right\vert ^{q}\right\} ^{1/q} \\
&&+\frac{1}{2}\left\{ \frac{1}{r+1}\sum_{\nu =\nu _{0}}^{r}\left\vert \frac{1%
}{\pi }\int_{2\delta _{\nu }}^{2\gamma }\varphi _{x}\left( t-\delta _{\nu
}\right) \left( \frac{1}{2\sin \frac{t}{2}}-\frac{1}{2\sin \frac{t-\delta
_{\nu }}{2}}\right) \sin \frac{\left( 2k_{\nu }+1\right) t}{2}dt\right\vert
^{q}\right\} ^{1/q}
\end{eqnarray*}%
\begin{eqnarray*}
&\ll &\left\{ \frac{1}{r+1}\sum_{\nu =\nu _{0}}^{r}\left\vert \frac{1}{\pi }%
\left[ \int_{0}^{t}\left( \varphi _{x}\left( u\right) -\varphi _{x}\left(
u-\delta _{\nu }\right) \right) \sin \frac{\left( 2k_{\nu }+1\right) u}{2}du%
\frac{1}{2\sin \frac{t}{2}}\right] _{t=2\delta _{\nu }}^{2\gamma }\right.
\right. \\
&&+\left. \left. \frac{1}{\pi }\int_{2\delta _{\nu }}^{2\gamma _{r}}\left[
\int_{0}^{t}\left( \varphi _{x}\left( u\right) -\varphi _{x}\left( u-\delta
_{\nu }\right) \right) \sin \frac{\left( 2k_{\nu }+1\right) u}{2}du\right] 
\frac{\cos \frac{t}{2}}{\left( 2\sin \frac{t}{2}\right) ^{2}}dt\right\vert
^{q}\right\} ^{1/q} \\
&&+\left\{ \frac{1}{r+1}\sum_{\nu =\nu _{0}}^{r}\left\vert \delta _{\nu }%
\frac{1}{\pi }\int_{2\delta _{\nu }}^{2\gamma _{r}}\frac{\left\vert \varphi
_{x}\left( t-\delta _{\nu }\right) \right\vert }{t^{2}}dt\right\vert
^{q}\right\} ^{1/q}
\end{eqnarray*}%
\begin{eqnarray*}
&\ll &\left\{ \frac{1}{r+1}\sum_{\nu =\nu _{0}}^{r}\left[ \left\vert \frac{1%
}{\pi }\int_{0}^{2\gamma }\left( \varphi _{x}\left( u\right) -\varphi
_{x}\left( u-\delta _{\nu }\right) \right) \sin \frac{\left( 2k_{\nu
}+1\right) u}{2}du\frac{1}{2\sin \frac{2\gamma }{2}}\right\vert \right.
\right. \\
&&+\left\vert \frac{1}{\pi }\int_{0}^{2\delta _{\nu }}\left( \varphi
_{x}\left( u\right) -\varphi _{x}\left( u-\delta _{\nu }\right) \right) \sin 
\frac{\left( 2k_{\nu }+1\right) u}{2}du\frac{1}{2\sin \frac{2\delta _{\nu }}{%
2}}\right\vert \\
&&+\left. \left. \frac{1}{\pi }\int_{2\delta _{\nu }}^{2\gamma }\left[
\int_{0}^{t}\left\vert \left( \varphi _{x}\left( u\right) -\varphi
_{x}\left( u-\delta _{\nu }\right) \right) \sin \frac{\left( 2k_{\nu
}+1\right) u}{2}\right\vert du\right] \frac{\pi ^{2}}{\left( 2t\right) ^{2}}%
dt\right] ^{q}\right\} ^{1/q} \\
&&+\left\{ \frac{1}{r+1}\sum_{\nu =\nu _{0}}^{r}\left[ \delta _{\nu }\frac{1%
}{\pi }\int_{\delta _{\nu }}^{2\gamma -\delta _{\nu }}\frac{\left\vert
\varphi _{x}\left( t\right) \right\vert }{\left( t+\delta _{\nu }\right) ^{2}%
}dt\right] ^{q}\right\} ^{1/q},
\end{eqnarray*}%
and applying Lemmas 1,2,3 we have%
\begin{eqnarray*}
&&B_{r,\nu _{0}}^{1} \\
&\ll &\left\{ \frac{1}{r+1}\sum_{\nu =\nu _{0}}^{r}\left[ \frac{1}{8\gamma }%
\int_{0}^{2\gamma }\left\vert \varphi _{x}\left( u\right) -\varphi
_{x}\left( u-\delta _{\nu }\right) \right\vert du\right. \right. \\
&&+\left. \left. \frac{1}{4\delta _{\nu }}\int_{0}^{2\delta _{\nu
}}\left\vert \varphi _{x}\left( u\right) -\varphi _{x}\left( u-\delta _{\nu
}\right) \right\vert du\right. \right. \\
&&+\left. \left. \frac{\pi }{8}\int_{2\delta _{\nu }}^{2\gamma }\left( \frac{%
1}{t^{2}}\int_{0}^{t}\left\vert \varphi _{x}\left( u\right) -\varphi
_{x}\left( u-\delta _{\nu }\right) \right\vert du\right) dt\right]
^{q}\right\} ^{1/q} \\
&&+\left\{ \frac{1}{r+1}\sum_{\nu =\nu _{0}}^{r}\left[ \delta _{\nu
}\int_{\delta _{\nu }}^{\pi }\frac{\left\vert \varphi _{x}\left( t\right)
\right\vert }{t^{2}}dt\right] ^{q}\right\} ^{1/q}
\end{eqnarray*}%
\begin{eqnarray*}
&\ll &w_{x}\left( \delta _{0}\right) \\
&&+\frac{\pi }{8}\left\{ \frac{1}{r+1}\sum_{\nu =\nu _{0}}^{r}\left[
\int_{2\delta _{\nu }}^{2\gamma }\frac{1}{t}w_{x}\left( \delta _{\nu
}\right) dt\right] ^{q}\right\} ^{1/q} \\
&&+\left\{ \frac{1}{r+1}\sum_{\nu =\nu _{0}}^{r}\left[ \delta _{\nu
}\sum_{\mu =0}^{k_{\nu }}w_{x}f\left( \frac{\pi }{\mu +1}\right) _{1}\right]
^{q}\right\} ^{1/q}
\end{eqnarray*}%
\begin{eqnarray*}
&\ll &w_{x}\left( \delta _{0}\right) +Kw_{x}\left( \delta _{0}\right) \log 
\frac{\gamma }{\delta _{r}}+K\delta _{0}\sum_{\mu =0}^{k_{0}}w_{x}\left( 
\frac{\pi }{\mu +1}\right) \\
&\leq &Kw_{x}\left( \delta _{0}\right) \left( 1+\log \frac{k_{r}+1/2}{r+1/2}%
\right) .
\end{eqnarray*}%
Consequently, by Lemma 4,

\begin{eqnarray*}
B_{r,\nu _{0}}^{2} &=&\frac{1}{2}\left\{ \frac{1}{r+1}\sum_{\nu =\nu
_{0}}^{r}\left\vert \frac{1}{\pi }\int_{2\gamma -\delta _{\nu }}^{2\gamma
}\varphi _{x}\left( t\right) D_{k_{\nu }}\left( t\right) dt\right\vert
^{q}\right\} ^{1/q} \\
&\leq &\frac{1}{2}\left\{ \frac{1}{r+1}\sum_{\nu =\nu _{0}}^{r}\left\vert 
\frac{1}{\pi }\int_{2\gamma -\delta _{\nu }}^{2\gamma }\left\vert \varphi
_{x}\left( t\right) \right\vert \frac{\pi }{2t}dt\right\vert ^{q}\right\}
^{1/q}
\end{eqnarray*}%
\begin{eqnarray*}
&\leq &\frac{1}{2}\left\{ \frac{1}{r+1}\sum_{\nu =\nu _{0}}^{r}\left\vert 
\frac{1}{\pi }\int_{2\gamma -\delta _{\nu }}^{2\gamma }\left\vert \varphi
_{x}\left( t\right) \right\vert \frac{\pi }{2t}dt\right\vert ^{q}\right\}
^{1/q} \\
&\leq &\frac{1}{4}\left\{ \frac{1}{r+1}\sum_{\nu =\nu _{0}}^{r}\left\vert
\int_{2\gamma -\delta _{\nu }}^{2\gamma }\frac{d}{dt}\left(
\int_{0}^{t}\left\vert \varphi _{x}\left( u\right) \right\vert du\right) 
\frac{dt}{t}\right\vert ^{q}\right\} ^{1/q}
\end{eqnarray*}%
\begin{eqnarray*}
&\leq &\frac{1}{4}\left\{ \frac{1}{r+1}\sum_{\nu =\nu _{0}}^{r}\left\vert %
\left[ \frac{1}{t}\int_{0}^{t}\left\vert \varphi _{x}\left( u\right)
\right\vert du\right] _{t=2\gamma -\delta _{\nu }}^{t=2\gamma
}+\int_{2\gamma -\delta _{\nu }}^{2\gamma }\frac{w_{x}\left( t\right) }{t}%
dt\right\vert ^{q}\right\} ^{1/q} \\
&\ll &\left\{ \frac{1}{r+1}\sum_{\nu =\nu _{0}}^{r}\left\vert \frac{1}{%
2\gamma }\int_{0}^{2\gamma }\left\vert \varphi _{x}\left( u\right)
\right\vert du-\frac{1}{2\gamma -\delta _{\nu }}\int_{0}^{2\gamma -\delta
_{\nu }}\left\vert \varphi _{x}\left( u\right) \right\vert du\right. \right.
\\
&&+\left. \left. \frac{w_{x}\left( 2\gamma -\delta _{\nu }\right) }{2\gamma
-\delta _{\nu }}\int_{2\gamma -\delta _{\nu }}^{2\gamma }dt\right\vert
^{q}\right\} ^{1/q}
\end{eqnarray*}%
\begin{eqnarray*}
&\ll &\left\{ \frac{1}{r+1}\sum_{\nu =\nu _{0}}^{r}\left\vert \frac{1}{%
2\gamma -\delta _{\nu }}\int_{0}^{2\gamma }\left[ \left\vert \varphi
_{x}\left( u\right) \right\vert -\left\vert \varphi _{x}\left( u-\delta
_{\nu }\right) \right\vert \right] du\right. \right. \\
&&+\left. \left. \frac{1}{2\gamma -\delta _{\nu }}\int_{0}^{\delta _{\nu
}}\left\vert \varphi _{x}\left( u-\delta _{\nu }\right) \right\vert du+\frac{%
w_{x}\left( \delta _{\nu }\right) }{\delta _{\nu }}\delta _{\nu }\right\vert
^{q}\right\} ^{1/q}
\end{eqnarray*}%
\begin{eqnarray*}
&\ll &\left\{ \frac{1}{r+1}\sum_{\nu =\nu _{0}}^{r}\left\vert \frac{1}{%
\gamma }\int_{0}^{2\gamma }\left[ \left\vert \varphi _{x}\left( u\right)
-\varphi _{x}\left( u-\delta _{\nu }\right) \right\vert \right] du\right.
\right. \\
&&+\left. \left. \frac{1}{\delta _{\nu }}\int_{-\delta _{\nu
}}^{0}\left\vert \varphi _{x}\left( u\right) \right\vert du+w_{x}\left(
\delta _{\nu }\right) \right\vert ^{q}\right\} ^{1/q} \\
&\ll &w_{x}\left( \delta _{0}\right)
\end{eqnarray*}%
and%
\begin{eqnarray*}
B_{r,\nu _{0}}^{3} &=&\frac{1}{2}\left\{ \frac{1}{r+1}\sum_{\nu =\nu
_{0}}^{r}\left\vert \frac{1}{\pi }\int_{\delta _{\nu }}^{2\delta _{\nu
}}\varphi _{x}\left( t\right) D_{k_{\nu }}\left( t\right) dt\right\vert
^{q}\right\} ^{1/q} \\
&\leq &\frac{1}{2}\left\{ \frac{1}{r+1}\sum_{\nu =\nu _{0}}^{r}\left\vert 
\frac{1}{\pi }\int_{\delta _{\nu }}^{2\delta _{\nu }}\left\vert \varphi
_{x}\left( t\right) \right\vert \frac{\pi }{2t}dt\right\vert ^{q}\right\}
^{1/q} \\
&\leq &\frac{1}{2}\left\{ \frac{1}{r+1}\sum_{\nu =\nu _{0}}^{r}\left[
w_{x}f\left( 2\delta _{\nu }\right) \right] ^{q}\right\} ^{1/q}\ll
w_{x}\left( \delta _{0}\right) .
\end{eqnarray*}%
Thus%
\begin{equation*}
B_{r}\ll w_{x}\left( \delta _{0}\right) \left( 1+\log \frac{k_{r}+1}{r+1/2}%
\right) .
\end{equation*}

Finally we estimate the term \ $C_{r}$ dividing it into the two parts. 
\begin{eqnarray*}
&&C_{r} \\
&=&\left\{ \frac{1}{r+1}\sum_{\nu =0}^{r}\left\vert \frac{1}{\pi }%
\int_{2\gamma }^{\pi }\varphi _{x}\left( t\right) \left( 2\sin \frac{t}{2}%
\right) ^{-1}\sin \left( \left( k_{\nu }+\frac{1}{2}\right) t\right)
dt\right\vert ^{q}\right\} ^{1/q} \\
&\leq &\left\{ \frac{1}{r+1}\sum_{\nu =0}^{r}\left\vert \frac{1}{\pi }%
\int_{2\gamma }^{\pi }\left[ \frac{\Phi _{x}f\left( \delta _{0},t\right)
-\varphi _{x}\left( t\right) }{2\sin \frac{t}{2}}\right] \sin \left( \left(
k_{\nu }+\frac{1}{2}\right) t\right) dt\right\vert ^{q}\right\} ^{1/q} \\
&&+\left\{ \frac{1}{r+1}\sum_{\nu =0}^{r}\left\vert \frac{1}{\pi }%
\int_{2\gamma }^{\pi }\frac{\Phi _{x}f\left( \delta _{0},t\right) }{2\sin 
\frac{t}{2}}\sin \left( \left( k_{\nu }+\frac{1}{2}\right) t\right)
dt\right\vert ^{q}\right\} ^{1/q} \\
&=&C_{r}^{1}+C_{r}^{2}.
\end{eqnarray*}

Integrating by\ parts and applying Lemma 4 we obtain%
\begin{eqnarray*}
C_{r}^{1} &\leq &\frac{1}{\delta _{0}}\int_{0}^{\delta _{0}}\left[
\int_{2\gamma }^{\pi }\frac{\left\vert \varphi _{x}\left( u+t\right)
-\varphi _{x}\left( t\right) \right\vert }{t}dt\right] du \\
&=&\frac{1}{\delta _{0}}\int_{0}^{\delta _{0}}\left[ \int_{2\gamma }^{\pi }%
\frac{1}{t}\frac{d}{dt}\left( \int_{0}^{t}\left\vert \varphi _{x}\left(
u+v\right) -\varphi _{x}\left( v\right) \right\vert dv\right) dt\right] du
\end{eqnarray*}%
\begin{eqnarray*}
&=&\frac{1}{\delta _{0}}\int_{0}^{\delta _{0}}\left\{ \left[ \frac{1}{t}%
\int_{0}^{t}\left\vert \varphi _{x}\left( u+v\right) -\varphi _{x}\left(
v\right) \right\vert dv\right] _{t=2\gamma }^{\pi }\right. \\
&&+\left. \int_{2\gamma }^{\pi }\left( \frac{1}{t^{2}}\int_{0}^{t}\left\vert
\varphi _{x}\left( u+v\right) -\varphi _{x}\left( v\right) \right\vert
dv\right) dt\right\} du
\end{eqnarray*}%
\begin{eqnarray*}
&\leq &\frac{1}{\delta _{0}}\int_{0}^{\delta _{0}}w_{x}\left( u\right) du+%
\frac{1}{\delta _{0}}\int_{0}^{\delta _{0}}\left\{ \int_{2\gamma }^{\pi }%
\frac{1}{t}w_{x}\left( u\right) dt\right\} du \\
&\leq &w_{x}\left( \delta _{0}\right) +w_{x}\left( \delta _{0}\right)
\int_{2\gamma }^{\pi }\frac{1}{t}dt \\
&\leq &w_{x}\left( \delta _{0}\right) \left( 1+\log \pi -\log \gamma \right)
\\
&\leq &w_{x}\left( \delta _{0}\right) \left( 1+\log \frac{k_{r}+1}{r+1}%
\right)
\end{eqnarray*}

and additionally by Lemma 6%
\begin{eqnarray*}
&&\ C_{r}^{2} \\
&=&\frac{1}{2\left( r+1\right) ^{1/q}}\left\{ \sum_{\nu =0}^{r}\left\vert 
\frac{1}{\pi }\int_{2\gamma }^{\pi }\frac{\Phi _{x}f\left( \delta
_{0},t\right) }{2\sin \frac{t}{2}}\frac{d}{dt}\left( \frac{\cos \left(
\left( k_{\nu }+\frac{1}{2}\right) t\right) }{k_{\nu }+\frac{1}{2}}\right)
dt\right\vert ^{q}\right\} ^{1/q}
\end{eqnarray*}%
\begin{eqnarray*}
&=&\frac{1}{2\pi \left( r+1\right) ^{1/q}}\left\{ \sum_{\nu
=0}^{r}\left\vert \left[ \frac{\Phi _{x}f\left( \delta _{0},t\right) }{2\sin 
\frac{t}{2}}\frac{\cos \left( \left( k_{\nu }+\frac{1}{2}\right) t\right) }{%
k_{\nu }+\frac{1}{2}}\right] _{2\gamma }^{\pi }\right. \right. \\
&&\left. \left. -\int_{2\gamma }^{\pi }\frac{d}{dt}\left( \frac{\Phi
_{x}f\left( \delta _{0},t\right) }{2\sin \frac{t}{2}}\right) \frac{\cos
\left( \left( k_{\nu }+\frac{1}{2}\right) t\right) }{k_{\nu }+\frac{1}{2}}%
dt\right\vert ^{q}\right\} ^{1/q}
\end{eqnarray*}%
\begin{eqnarray*}
&\leq &\frac{1}{2\pi \left( r+1\right) ^{1/q}}\left\{ \sum_{\nu =0}^{r}\left[
\left\vert \frac{\Phi _{x}f\left( \delta _{0},2\gamma \right) }{2\sin \gamma 
}\frac{\cos \left( \left( k_{\nu }+\frac{1}{2}\right) 2\gamma \right) }{%
k_{\nu }+\frac{1}{2}}\right\vert \right. \right. \\
&&\left. \left. +\left\vert \int_{2\gamma }^{\pi }\frac{d}{dt}\left( \frac{%
\Phi _{x}f\left( \delta _{0},t\right) }{2\sin \frac{t}{2}}\right) \frac{\cos
\left( \left( k_{\nu }+\frac{1}{2}\right) t\right) }{k_{\nu }+\frac{1}{2}}%
dt\right\vert \right] ^{q}\right\} ^{1/q}
\end{eqnarray*}

\begin{eqnarray*}
&\leq &\frac{\left\vert \Phi _{x}f\left( \delta _{0},2\gamma \right)
\right\vert }{\gamma \left( k_{0}+1\right) }+\frac{1}{k_{0}+1}\int_{2\gamma
}^{\pi }\left\vert \frac{d}{dt}\left( \frac{\Phi _{x}f\left( \delta
_{0},t\right) }{2\sin \frac{t}{2}}\right) \right\vert dt \\
&\leq &\frac{1}{\gamma \left( k_{0}+1\right) }\frac{1}{\delta _{0}}%
\int_{0}^{\delta _{0}}\left\vert \varphi _{x}\left( u+2\gamma \right)
\right\vert du+\delta _{0}\int_{2\gamma }^{\pi }\frac{\left\vert \varphi
_{x}\left( \delta _{0}+t\right) -\varphi _{x}\left( t\right) \right\vert }{%
\delta _{0}t}dt \\
&&+\frac{1}{\delta _{0}}\int_{0}^{\delta _{0}}\left( \delta
_{0}\int_{2\gamma }^{\pi }\frac{\left\vert \varphi _{x}\left( u+t\right)
\right\vert }{t^{2}}dt\right) du
\end{eqnarray*}%
\begin{eqnarray*}
&\leq &\left\vert \Psi _{x}f\left( \delta _{0},2\gamma \right) \right\vert
+\int_{2\gamma }^{\pi }\frac{1}{t}\frac{d}{dt}\left( \int_{0}^{t}\left\vert
\varphi _{x}\left( \delta _{0}+u\right) -\varphi _{x}\left( u\right)
\right\vert du\right) dt \\
&&+\frac{1}{\delta _{0}}\int_{0}^{\delta _{0}}\left( \delta
_{0}\int_{2\gamma }^{\pi }\frac{\left\vert \varphi _{x}\left( u+t\right)
\right\vert }{t^{2}}dt\right) du \\
&\leq &G_{x}f\left( \delta _{0}\right) _{1,s}+\left[ \frac{1}{t}%
\int_{0}^{t}\left\vert \varphi _{x}\left( \delta _{0}+u\right) -\varphi
_{x}\left( u\right) \right\vert du\right] _{t=2\gamma }^{\pi } \\
&&+\int_{2\gamma }^{\pi }\frac{w_{x}\left( \delta _{0}\right) }{t}dt++\frac{1%
}{\delta _{0}}\int_{0}^{\delta _{0}}\left( \delta _{0}\int_{2\gamma }^{\pi }%
\frac{\left\vert \varphi _{x}\left( u+t\right) \right\vert }{t^{2}}dt\right)
du
\end{eqnarray*}%
\begin{equation*}
\leq w_{x}\left( \delta _{0}\right) \left( 1+\int_{2\gamma }^{\pi }\frac{1}{t%
}dt\right) \leq w_{x}\left( \delta _{0}\right) \left( 1+\log \frac{k_{r}+1}{%
r+1}\right) .
\end{equation*}%
Collecting \ our estimates we obtain desired estimate. $\ \blacksquare $

\subsection{Proof of Theorem 3}

If \ $w_{x}\left( \delta \right) \equiv 0$ \ then \ $f$ \ is constant and
our inequality is true. Thus we can suppose that \ \ $w_{x}\left( \delta
\right) >0$ \ \ for \ $\delta >0$.

Let denote by 
\begin{eqnarray*}
\Delta _{\mu } &=&\left\{ \nu :\left\vert S_{\nu }f\left( x\right) -f\left(
x\right) \right\vert \geq \mu w_{x}\left( u\right) \right\} \\
\Gamma _{\mu } &=&\left\{ \nu :\left( \mu -1\right) G_{x}^{\circ }f\left(
u\right) _{1,s}\leq \left\vert S_{\nu }f\left( x\right) -f\left( x\right)
\right\vert \leq \mu w_{x}\left( u\right) \right\} \\
\Theta &=&\left\{ \mu :\Gamma _{\mu }\neq \varnothing \right\}
\end{eqnarray*}%
the sets of integers \ $\nu \in \left[ N_{m-2}+1,N_{m}\right] $ \ and \ $\mu
,$ where \ $u=$\ $\frac{\pi }{N_{m-2}+1}$., \ then%
\begin{eqnarray*}
H_{m}^{\lambda \varphi }f\left( x\right) &\leq &\frac{1}{N_{m}+1}\sum_{\mu
\in \Theta }\sum_{\nu \in \Gamma _{\mu }}\varphi \left( \left\vert S_{\nu
}f\left( x\right) -f\left( x\right) \right\vert \right) \\
&\leq &\frac{1}{N_{m}+1}\sum_{\mu \in \Theta }\sum_{\nu \in \Gamma _{\mu
}}\varphi \left( \mu w_{x}\left( u\right) \right) \\
&=&\frac{1}{N_{m}+1}\sum_{\mu \in \Theta }\left\vert \Gamma _{\mu
}\right\vert \varphi \left( \mu w_{x}\left( u\right) \right) \\
&\leq &.\frac{1}{N_{m}+1}\sum_{\mu \in \Theta }\left\vert \Delta _{\mu
-1}\right\vert \varphi \left( \mu w_{x}\left( u\right) \right) .
\end{eqnarray*}%
Using Theorem 1 we can compute that \ \ $\left\vert \Delta _{\mu
-1}\right\vert \leq N_{m}\exp \left( -\frac{\mu -1}{K}\right) $,\ whence%
\begin{eqnarray*}
H_{m}^{\lambda \varphi }f\left( x\right) &\leq &\frac{1}{N_{m}+1}\sum_{\mu
\in \Theta }N_{m}\exp \left( -\frac{\mu -1}{K}\right) \varphi \left( \mu
w_{x}\left( u\right) \right) \\
&\ll &\sum_{\mu \in \Theta }\exp \left( -\frac{\mu }{K}\right) \varphi
\left( \mu w_{x}\left( u\right) \right) .
\end{eqnarray*}%
Since $\varphi \in \Phi $ , \ we have%
\begin{eqnarray*}
H_{m}^{\lambda \varphi }f\left( x\right) &\ll &\varphi \left( w_{x}\left(
u\right) \right) \\
&&+\left( \sum_{n=0}^{n_{0}}+\sum_{n=n_{0}+1}^{\infty }\right) \sum_{\mu
=2^{n}+1}^{2^{n+1}}\exp \left( -\frac{\mu }{K}\right) \varphi \left( \mu
w_{x}\left( u\right) \right) \\
&\ll &\varphi \left( w_{x}\left( u\right) \right) +\sum_{n=0}^{\infty
}\sum_{\mu =2^{n}+1}^{2^{n+1}}\exp \left( -\frac{2^{n}}{K}\right) \varphi
\left( 2^{n+1}w_{x}\left( u\right) \right) \\
&\ll &\varphi \left( w_{x}\left( u\right) \right) +\sum_{n=0}^{\infty
}2^{n}\exp \left( -\frac{2^{n}}{K}\right) \varphi \left( 2^{n}w_{x}\left(
u\right) \right) \\
&\ll &\varphi \left( w_{x}\left( u\right) \right)
+\sum_{n=0}^{n_{0}}2^{n}\exp \left( -\frac{2^{n}}{K}\right) \varphi \left(
2^{n}w_{x}\left( u\right) \right) \\
&&+\sum_{n=n_{0}+1}^{\infty }2^{n}\exp \left( -\frac{2^{n}}{K}\right)
\varphi \left( 2^{n}w_{x}\left( u\right) \right) \\
&\ll &\varphi \left( w_{x}\left( u\right) \right)
\end{eqnarray*}%
with some \ $n_{0}$\ , analogously as in.\cite{12} p.108, \ and therefore
our proof is complete.

$\blacksquare $

\subsection{Proof of Theorem 5}

We start with the obvious inequality%
\begin{equation*}
H_{\cdot }^{\lambda \varphi }f\left( x\right) \ll \sum_{m=2}^{\infty
}\sum_{\nu =N_{m-2}+1}^{N_{m}}\lambda _{\nu }\varphi \left( \left\vert
S_{\nu }f\left( x\right) -f\left( x\right) \right\vert \right) .
\end{equation*}%
Using the H\"{o}lder inequality we obtain%
\begin{equation*}
H_{\cdot }^{\lambda \varphi }f\left( x\right) \ll \sum_{m=1}^{\infty
}\left\{ \sum_{\nu =N_{m-2}+1}^{N_{m}}\left( \lambda _{\nu }\right)
^{s}\right\} ^{1/s}\left\{ \sum_{\nu =N_{m-2}+1}^{N_{m}}\varphi ^{q}\left(
\left\vert S_{\nu }f\left( x\right) -f\left( x\right) \right\vert \right)
\right\} ^{1/q}
\end{equation*}%
with \ \ $\frac{1}{s}+\frac{1}{q}=1$ \ $\left( s>1\right) ,$ \ and by the
assumption \ $\left( \lambda _{\nu }\right) \in \Lambda _{s}\left(
N_{m}\right) ,$we have \ 
\begin{equation*}
H_{\cdot }^{\lambda \varphi }f\left( x\right) \ll \sum_{m=1}^{\infty
}\sum_{\nu =N_{m-2}+1}^{N_{m}}\lambda _{\nu }\left\{ \frac{1}{N_{m}}%
\sum_{\nu =N_{m-2}+1}^{N_{m}}\varphi ^{q}\left( \left\vert S_{\nu }f\left(
x\right) -f\left( x\right) \right\vert \right) \right\} ^{1/q}.
\end{equation*}%
The second assumption \ $\varphi \in \Phi $ \ also implies\ that $\varphi
^{q}\in \Phi .$ and therefore, by the Theorem 3, 
\begin{eqnarray*}
H_{\cdot }^{\lambda \varphi }f\left( x\right) &\ll &\sum_{m=1}^{\infty
}\sum_{\nu =N_{m-2}+1}^{N_{m}}\lambda _{\nu }\left\{ \varphi ^{q}\left(
w_{x}\left( \frac{\pi }{N_{m-2}+2}\right) \right) \right\} ^{1/q} \\
&\ll &\sum_{m=1}^{\infty }\sum_{\nu =N_{m-2}+1}^{N_{m}}\lambda _{\nu
}\varphi \left( w_{x}\left( \frac{\pi }{N_{m-2}+2}\right) \right) .
\end{eqnarray*}

Thus our result is proved. $\ \ \blacksquare $

\end{document}